\documentclass[a4paper,12pt]{amsart}
\usepackage{amsmath,amsfonts,amssymb}
\usepackage{amsthm}
\usepackage{amstext,amscd,mathabx,mathdots}
\usepackage{graphics,graphicx,epsf}
\usepackage[usenames]{color}
\usepackage{hyperref}
\usepackage{tikz}
\usepackage{xcolor}
\usepackage{color}
\usepackage{float}
\usetikzlibrary{arrows}
\usepackage{lineno}

\theoremstyle{emsthmsl}

\newtheorem{theorem}{Theorem}[section]
\newtheorem{definition}[theorem]{Definition}
\newtheorem{proposition}[theorem]{Proposition}
\newtheorem{lemma}[theorem]{Lemma}
\newtheorem{corollary}[theorem]{Corollary}
\newtheorem{remark}[theorem]{Remark}

\numberwithin{equation}{section}

\title[Euler-Rodrigues formula]{Euler-Rodrigues formula for three-dimensional rotation via fractional powers of matrices} 

\author[F. D. M. Bezerra]{Flank D. M. Bezerra}
\address[F. D. M. Bezerra]{Departamento de Matem\'atica, Universidade Federal da Para\'iba, 58051-900 Jo\~ao Pessoa PB, Brazil.}
\email{flank@mat.ufpb.br}

\author[L. A. Santos]{Lucas A. Santos}
\address[L. A. Santos]{Instituto Federal da Para\'iba, 58780-000 Itaporanga PB, Brazil.}
\email{lucas92mat@gmail.com}

\date{\today}

\begin{document}

\maketitle

\begin{abstract}
In this short paper, we review the Euler-Rodrigues formula for three-dimensional rotation via fractional powers of matrices. We derive the rotations by any angle through the spectral behavior of the fractional powers of the rotation matrix by $\frac{\pi}{2}$ in $\mathbb{R}^3$ about some axis.


\end{abstract}

\tableofcontents

\section{Introduction}

The Euler-Rodrigues formula describes the rotation of a vector in three dimensions, it was first discovered by Euler \cite{E} and later rediscovered independently by Rodrigues \cite{R} and it is related to a number of interesting problems in computer graphics, dynamics, kinematics, mathematics, and robotics, see Cheng and  Gupta \cite{CG} and references therein.

Reviews of the Euler–Rodrigues formula in different mathematical forms can be found in the literature, see e.g., Dai \cite{D}, Kahvec\'i, Yayli and G\"ok \cite{KYG} and Mebius \cite{M}. Here, we explored the geometric aspect of the classical Balakrishnan formula in \cite{B} to obtain a new algorithm for the generation of a three-dimensional rotation matrix. 

To our best knowledge, this treatment on the Euler–Rodrigues formula has not yet been explored in the literature.

\section{Three-dimensional rotations}

Firstly, we present some facts of the theory of fractional powers of matrices. Secondly, we establish the main results of this paper; namely, we review the Euler-Rodrigues formula via the Balakrishnan formula on fractional powers of matrices.

\subsection{Fractional powers of operators}

In this subsection, we recall some definitions and summarize without proofs results of the theory of fractional powers of matrices, in the sense of Balakrishnan \cite{B}.

\begin{definition}
For $A\in\mathbb{C}^{n\times n}$ with no eigenvalues on $(-\infty,0)$ and $\alpha\in\mathbb{R}$, $A^\alpha=e^{\alpha\log A}$, where $\log A$ is the principal logarithm.
\end{definition}

Thanks to Balakrishnan \cite{B} we following results are well-known. 

\begin{proposition}
Let $0<\alpha<1$. We have
\begin{itemize}
\item[$(i)$] 
\begin{equation}\label{fracpos}
A^{\alpha}=\frac{\sin{\alpha\pi}}{\pi}\int_{0}^{\infty}{\lambda}^{\alpha-1}A(\lambda I+A)^{-1}d\lambda;
\end{equation}
\item[$(ii)$] Let $\beta$ be real number, then
\[
(A^{\alpha})^{\beta}=A^{\alpha\beta}.
\]
\end{itemize}
\end{proposition}

\subsection{Main results}

In this subsection, we present the main results of this paper. We explored the geometric aspect of the classical Balakrishnan formula \eqref{fracpos} (see, e.g., Balakrishnan \cite{B}) to obtain a new algorithm for the generation of three-dimensional rotation matrices. Here, the matrix representations of linear operators on $\mathbb{R}^3$ are considered using the standard basis of $\mathbb{R}^3$, and $\mathbf{\hat{n}}=(n_1,n_2,n_3)$ denotes a vector in $\mathbb{R}^3$ with $n_1^2+n_2^2+n_3^2=1$.

\begin{lemma}\label{lemmapi2}
The matrix which represents the rotation by an angle $\frac{\pi}{2}$ about the axis $\mathbf{\hat{n}}=(n_1,n_2,n_3)$ is given by
\begin{equation}\label{e-Rot}
A\Big(\mathbf{\hat{n}},\frac{\pi}{2}\Big)=\left[\begin{matrix}
n_1^2 & n_1n_2-n_3 & n_1n_3+n_2 \\ n_1n_2+n_3 & n_2^2 & n_2n_3-n_1 \\ n_1n_3-n_2 & n_2n_3+n_1 & n_3^2
\end{matrix}\right].
\end{equation}
\end{lemma}
\begin{proof}
Choose two vectors, $\mathbf{\hat{l}}$ and $\mathbf{\hat{m}}$, such that $\{\mathbf{\hat{l}},\mathbf{\hat{m}},\mathbf{\hat{n}}\}$ is a right-handed orthonormal basis. Let $u=a\mathbf{\hat{l}}+b\mathbf{\hat{m}}+c\mathbf{\hat{n}}$, with $a,b,c\in\mathbb{R}$, be any vector to be rotated by an angle $\frac{\pi}{2}$ counterclockwise about the axis $\mathbf{\hat{n}}$. The resulting vector $u'$ is the vector $u$ with its component in the $\mathbf{\hat{l}},\mathbf{\hat{m}}$ plane rotated by $\frac{\pi}{2}$
\begin{eqnarray*}
u'&=& -b\mathbf{\hat{l}}+a\mathbf{\hat{m}}+c\mathbf{\hat{n}}\\
  &=& \mathbf{\hat{n}}\times u + \langle u,\mathbf{\hat{n}} \rangle \mathbf{\hat{n}}.
\end{eqnarray*}
Consider the standard basis $\{{\mathbf{\hat{e_1}},\mathbf{\hat{e_2}},\mathbf{\hat{e_3}}}\}$ of $\mathbb{R}^3$. If $u$ is written as
\begin{equation*}
u=u_1\mathbf{\hat{e_1}} + u_2\mathbf{\hat{e_2}} + u_3\mathbf{\hat{e_3}},
\end{equation*}
then
\begin{eqnarray*}
u'&=& \mathbf{\hat{n}}\times u + \langle u,\mathbf{\hat{n}} \rangle \mathbf{\hat{n}}\\
  &=& (n_2u_3-n_3u_2 + u_1n_1^2+u_2n_1n_2+u_3n_1n_3)\mathbf{\hat{e_1}}\\
  &+&(n_3u_1-n_1u_3+u_1n_1n_2+u_2n_2^2+u_3n_2n_3)\mathbf{\hat{e_2}}\\
  &+&(n_1u_2-n_2u_1 + u_1n_1n_3+u_2n_2n_3+u_3n_3^2)\mathbf{\hat{e_3}}.
\end{eqnarray*}

Therefore, the matrix representation of this rotation  is 
\[
A\Big(\mathbf{\hat{n}},\frac{\pi}{2}\Big)=\left[\begin{matrix}
n_1^2 & n_1n_2-n_3 & n_1n_3+n_2 \\ n_1n_2+n_3 & n_2^2 & n_2n_3-n_1 \\ n_1n_3-n_2 & n_2n_3+n_1 & n_3^2
\end{matrix}\right].  \square
\]
\end{proof}

\begin{remark}
Thanks to the characterization in \eqref{e-Rot} of the matrix which represents the rotation by an angle $\frac{\pi}{2}$ about the axis $\mathbf{\hat{n}}=(n_1,n_2,n_3)$ we can obtain a matrix characterization of the linear semigroup generated by $A\Big(\mathbf{\hat{n}},\frac{\pi}{2}\Big)$, namely the uniformly continuous semigroup of bounded linear operators generated by $A\Big(\mathbf{\hat{n}},\frac{\pi}{2}\Big)$, denoted by $T(\cdot)$,
has the following explicit representation
\begin{eqnarray*}
&&T(t)=e^{tA(\mathbf{\hat{n}},\frac{\pi}{2})}=\sum_{n=0}^{\infty}\frac{(tA(\mathbf{\hat{n}},\frac{\pi}{2}))^n}{n!}=\\
&&\left[\begin{matrix}
n_1^2(e^t-\cos t)+\cos t & n_1n_2(e^t-\cos t)-n_3\sin t & n_1n_3(e^t-\cos t)+n_2\sin t 
\\ n_1n_2(e^t-\cos t)+n_3\sin t & n_2^2(e^t-\cos t)+\cos t & n_2n_3(e^t-\cos t)-n_1\sin t 
\\ n_1n_3(e^t-\cos t)-n_2\sin t & n_2n_3(e^t-\cos t)+n_1\sin t & n_3^2(e^t-\cos t)+\cos t
\end{matrix}\right]
\end{eqnarray*}
for any $t\geqslant0$.
\end{remark}

\begin{remark}
An explicit formula for the matrix elements of a general $3\times3$ rotation matrix can be find in Rodrigues \cite{R}; namely, if $R(\mathbf{\hat{n}},\theta)$ denotes the a rotation by an angle $\theta$ about an axis $\mathbf{\hat{n}}=(n_1,n_2,n_3)$ $(n_1^2+n_2^2+n_3^2=1)$, whose elements are denoted by $R_{ij}(\mathbf{\hat{n}},\theta)$, then we have the Rodrigues formula 
\begin{equation}\label{e-FR}
R_{ij}(\mathbf{\hat{n}},\theta)=\cos(\theta)\delta_{ij}+(1-\cos(\theta))n_in_j-\sin(\theta)\epsilon_{ijk}n_k,
\end{equation}
where $\delta_{ij}$ denotes the Kronecker delta, i.e., 
\[
\delta_{ij}= 
\begin{cases}
1, &\mbox{if}\ i=j,\\
0, &\mbox{if}\ i\neq j, 
\end{cases}
\]  
and $\epsilon_{ijk}$ denotes the Levi-Civita tensor, i.e., 
\[
\epsilon_{ijk}= 
\begin{cases}
1, &\mbox{if}\ (i,j,k)\in\{(1,2,3),(2,3,1),(3,1,2)\},\\
-1, &\mbox{if}\ (i,j,k)\in\{(3,2,1),(1,3,2),(2,1,3)\},\\
0, &\mbox{if}\ i=j,\ \mbox{or}\ j=k,\ \mbox{or}\ k=i, 
\end{cases}
\]
which is called the angle-and-axis parameterization of the three-dimensional rotation matrix.
\end{remark}

We wish to derive all the rotations by any angle $\theta\in\mathbb{R}$ through the rotation by $\frac{\pi}{2}$ and its fractional powers. In order to get this result we first explicit, in the following theorem, the fractional power, for $0\leq\alpha\leq 1$, of the rotation $A(\mathbf{\hat{n}},\frac{\pi}{2})$ in Lemma \ref{lemmapi2}. It is one of the main results of this work.

\begin{theorem}\label{teofrac}
Let $A(\mathbf{\hat{n}},\frac{\pi}{2})$ be the matrix that represents the rotation by an angle $\frac{\pi}{2}$ about the axis $\mathbf{\hat{n}}=(n_1,n_2,n_3)$. 
For $0\leqslant \alpha \leqslant 1$, the fractional power of the rotation $A(\mathbf{\hat{n}},\frac{\pi}{2})$ is given by
\[
\begin{split}
&A^{\alpha}\Big(\mathbf{\hat{n}},\frac{\pi}{2}\Big)=\\
&\left[\begin{matrix}
n_1^2(1-\cos(\frac{\alpha\pi}{2}))+\cos(\frac{\alpha\pi}{2}) & n_1n_2(1-\cos(\frac{\alpha\pi}{2}))-n_3\sin(\frac{\alpha\pi}{2}) & n_1n_3(1-\cos(\frac{\alpha\pi}{2})))+n_2\sin(\frac{\alpha\pi}{2})\\n_1n_2(1-\cos(\frac{\alpha\pi}{2}))+n_3\sin(\frac{\alpha\pi}{2}) & n_2^2(1-cos\frac{\alpha\pi}{2})+cos\frac{\alpha\pi}{2} & n_2n_3(1-cos\frac{\alpha\pi}{2})-n_1\sin(\frac{\alpha\pi}{2})\\ n_1n_3(1-cos\frac{\alpha\pi}{2})-n_2\sin(\frac{\alpha\pi}{2}) & n_2n_3(1-\cos(\frac{\alpha\pi}{2}))+n_1\sin(\frac{\alpha\pi}{2}) & n_3^2(1-cos\frac{\alpha\pi}{2})+\cos(\frac{\alpha\pi}{2})
\end{matrix}\right].
\end{split}
\]
\end{theorem}

\begin{proof}
The proof consists of the explicit calculation of the fractional power of the operator $A(\mathbf{\hat{n}},\frac{\pi}{2})$ through the formula \eqref{fracpos} for $0<\alpha<1$.
\begin{equation}\label{potR}
A\Big(\mathbf{\hat{n}},\frac{\pi}{2}\Big)^{\alpha}=\frac{\sin(\alpha\pi)}{\pi}\int_{0}^{\infty}{\lambda}^{\alpha-1}A\Big(\mathbf{\hat{n}},\frac{\pi}{2}\Big)\Big(\lambda I+A\Big(\mathbf{\hat{n}},\frac{\pi}{2}\Big)\Big)^{-1}d\lambda,\ 0<\alpha<1.
\end{equation} 

Note that
\begin{eqnarray*}
&&\Big(\lambda I+A\Big(\mathbf{\hat{n}},\frac{\pi}{2}\Big)\Big)^{-1}=\\
&&\frac{1}{(\lambda+1)(\lambda^2+1)}\left[\begin{matrix}
a^2(1-\lambda)+\lambda(1+\lambda) & ab(1-\lambda)+c(1+\lambda)        & ac(1-\lambda)-b(1+\lambda)\\
ab(1-\lambda)-c(1+\lambda)        & b^2(1-\lambda)+\lambda(1+\lambda) & bc(1-\lambda)+a(1+\lambda)\\
ac(1-\lambda)+b(1+\lambda)        & bc(1-\lambda)-a(1+\lambda)        & c^2(1-\lambda)+\lambda(1+\lambda)
\end{matrix}\right]
\end{eqnarray*}
and
\begin{eqnarray*}
&&A\Big(\mathbf{\hat{n}},\frac{\pi}{2}\Big)\Big(\lambda I+A\Big(\mathbf{\hat{n}},\frac{\pi}{2}\Big)\Big)^{-1}=\\
&&\frac{1}{(\lambda+1)(\lambda^2+1)}\left[\begin{matrix}
a^2\lambda(\lambda-1)+1+\lambda            & ab\lambda(\lambda-1)-c\lambda(1+\lambda)    & ac\lambda(\lambda-1)+b\lambda(1+\lambda)\\
ab\lambda(\lambda-1)+c\lambda(1+\lambda)   & b^2\lambda(\lambda-1)+1+\lambda             & bc\lambda(\lambda-1)-a\lambda(1+\lambda)\\
ac\lambda(\lambda-1)-b\lambda(1+\lambda)   & bc\lambda(\lambda-1)+a\lambda(1+\lambda)    & c^2\lambda(\lambda-1)+1+\lambda
\end{matrix}\right].
\end{eqnarray*}

Since
\begin{eqnarray*}
\frac{\lambda(\lambda-1)}{(\lambda+1)(\lambda^2+1)}&=&\frac{1}{\lambda+1}-\frac{1}{\lambda^2+1}\\
\frac{\lambda+1}{(\lambda+1)(\lambda^2+1)}&=&\frac{1}{\lambda^2+1}\\
\frac{\lambda(\lambda+1)}{(\lambda+1)(\lambda^2+1)}&=&\frac{\lambda}{\lambda^2+1}\\
\end{eqnarray*}
from right side of the equation \eqref{potR} and $A\Big(\mathbf{\hat{n}},\frac{\pi}{2}\Big)\Big(\lambda I+A\Big(\mathbf{\hat{n}},\frac{\pi}{2}\Big)\Big)^{-1}$, and by  \eqref{fracpos}   we obtain
\[
\begin{split}
&A^{\alpha}\Big(\mathbf{\hat{n}},\frac{\pi}{2}\Big)=\\
&\left[\begin{matrix}
n_1^2(1-\cos(\frac{\alpha\pi}{2}))+\cos(\frac{\alpha\pi}{2}) & n_1n_2(1-\cos(\frac{\alpha\pi}{2}))-n_3\sin(\frac{\alpha\pi}{2}) & n_1n_3(1-\cos(\frac{\alpha\pi}{2})))+n_2\sin(\frac{\alpha\pi}{2})\\n_1n_2(1-\cos(\frac{\alpha\pi}{2}))+n_3\sin(\frac{\alpha\pi}{2}) & n_2^2(1-cos\frac{\alpha\pi}{2})+cos\frac{\alpha\pi}{2} & n_2n_3(1-cos\frac{\alpha\pi}{2})-n_1\sin(\frac{\alpha\pi}{2})\\ n_1n_3(1-cos\frac{\alpha\pi}{2})-n_2\sin(\frac{\alpha\pi}{2}) & n_2n_3(1-\cos(\frac{\alpha\pi}{2}))+n_1\sin(\frac{\alpha\pi}{2}) & n_3^2(1-cos\frac{\alpha\pi}{2})+\cos(\frac{\alpha\pi}{2})
\end{matrix}\right].
\end{split}
\]
Finally, cases $\alpha=0$ and $\alpha=1$ are immediate, and the proof is complete.   $\square$
\end{proof}

\begin{corollary}
The fractional power $A^{\alpha}(\mathbf{\hat{n}},\frac{\pi}{2})$ coincides with matrix $R(\mathbf{\hat{n}},\frac{\alpha\pi}{2})=[R_{ij}(\mathbf{\hat{n}},\frac{\alpha\pi}{2})]$, where $R_{ij}(\mathbf{\hat{n}},\frac{\alpha\pi}{2})$ is given by \eqref{e-FR}, for $0\leqslant \alpha \leqslant 1$. 
\end{corollary}

We are now in a position to give our definition for the rotation matrix by an angle $\theta$ through fractional powers of the rotation by $\frac{\pi}{2}$.

\begin{definition}
The rotation by $\theta\in\mathbb{R}$, denoted by $A(\mathbf{\hat{n}},\theta)$, is defined to be
\begin{equation}\label{eqrot}
A(\mathbf{\hat{n}},\theta):=A^{\frac{2\theta}{\pi}}\Big(\mathbf{\hat{n}},\frac{\pi}{2}\Big).
\end{equation}
\end{definition}

Note that $A(\mathbf{\hat{n}},\frac{\pi}{2})$ is such that the fractional power $A^{\alpha}(\mathbf{\hat{n}},\frac{\pi}{2})$ is well-defined for $\alpha\in \mathbb{R}$. Theorem \ref{teofrac} states that the definition in \eqref{eqrot} agrees with the classical one given by \textit{Rodrigues formula} in \eqref{e-FR} for $0\leq\theta\leq \frac{\pi}{2}$. The following theorem extends this result for $\theta\in \mathbb{R}$. 

\begin{theorem}
Let $A(\mathbf{\hat{n}},\theta)$ be the rotation defined in \eqref{eqrot}. Then 
\begin{equation}\label{eqigrod}
A(\mathbf{\hat{n}},\theta)=R(\mathbf{\hat{n}},\theta)
\end{equation}
for any $\theta\in\mathbb{R}$.   
\end{theorem}

\begin{proof} Firstly for $\theta\geqslant0$, it is sufficient to show that \eqref{eqigrod} is satisfied for 
\[
\frac{(n-1)\pi}{2}\leq \theta \leq \frac{n\pi}{2},
\]
for $n\in\mathbb{N}$. We proceed by induction. The case $n=1$ follows from Theorem \ref{teofrac}. If we assume \eqref{eqigrod} for $n$, we can prove the result for $n+1$. Set 
$$
\frac{n\pi}{2}\leq \theta \leq \frac{(n+1)\pi}{2}
$$
so that
\begin{equation*}
\frac{(n-1)\pi}{2}\leq \theta-\frac{\pi}{2} \leq \frac{n\pi}{2}
\end{equation*}

Hence
\begin{equation}\label{AA}
A(\mathbf{\hat{n}},\theta)=A^{\frac{2\theta}{\pi}}\Big(\mathbf{\hat{n}},\frac{\pi}{2}\Big)=A^{\frac{2\theta}{\pi}-1}\Big(\mathbf{\hat{n}},\frac{\pi}{2}\Big)A\Big(\mathbf{\hat{n}},\frac{\pi}{2}\Big)=A\Big(\mathbf{\hat{n}},\theta-\frac{\pi}{2})A\Big(\mathbf{\hat{n}},\frac{\pi}{2}\Big)
\end{equation}
and by induction hypothesis 
\begin{equation}\label{indhyp}
A\Big(\mathbf{\hat{n}},\theta-\frac{\pi}{2}\Big)=R\Big(\mathbf{\hat{n}},\theta-\frac{\pi}{2}\Big)
\end{equation}
combining \eqref{AA} with \eqref{indhyp} we obtain
\begin{eqnarray*}\label{ARA}
A(\mathbf{\hat{n}},\theta)&=&R\Big(\mathbf{\hat{n}},\theta-\frac{\pi}{2}\Big)A\Big(\mathbf{\hat{n}},\frac{\pi}{2}\Big)\\
                          &=&R(\mathbf{\hat{n}},\theta)R\Big(\mathbf{\hat{n}},-\frac{\pi}{2}\Big)R\Big(\mathbf{\hat{n}},\frac{\pi}{2}\Big)\\
                          &=&R(\mathbf{\hat{n}},\theta)
\end{eqnarray*}
above we use some basic properties of the Euler-Rodrigues formula.  

Secondly, for $-\frac{\pi}{2}\leqslant\theta\leqslant0$, and proceeding analogously to the proof of Theorem \eqref{teofrac} we can obtain the expression
\[
\begin{split}
&A^{-\alpha}\Big(\mathbf{\hat{n}},\frac{\pi}{2}\Big)=\\
&\left[\begin{matrix}
n_1^2(1-\cos(\frac{\alpha\pi}{2}))+\cos(\frac{\alpha\pi}{2}) & n_1n_2(1-\cos(\frac{\alpha\pi}{2}))+n_3\sin(\frac{\alpha\pi}{2}) & n_1n_3(1-\cos(\frac{\alpha\pi}{2})))-n_2\sin(\frac{\alpha\pi}{2})\\n_1n_2(1-\cos(\frac{\alpha\pi}{2}))-n_3\sin(\frac{\alpha\pi}{2}) & n_2^2(1-\cos(\frac{\alpha\pi}{2}))+\cos(\frac{\alpha\pi}{2}) & n_2n_3(1-\cos(\frac{\alpha\pi}{2}))+n_1\sin(\frac{\alpha\pi}{2})\\ n_1n_3(1-\cos(\frac{\alpha\pi}{2}))+n_2\sin(\frac{\alpha\pi}{2}) & n_2n_3(1-\cos(\frac{\alpha\pi}{2}))-n_1\sin(\frac{\alpha\pi}{2}) & n_3^2(1-\cos(\frac{\alpha\pi}{2}))+\cos(\frac{\alpha\pi}{2})
\end{matrix}\right]
\end{split}
\]
and so the definition in \eqref{eqrot} agrees with the classical one given by the Euler-Rodrigues formula in \eqref{e-FR} for $-\frac{\pi}{2}\leqslant\theta\leqslant0$. Finally, an analogous argument of induction as in the first part of this proof shows that \eqref{eqrot} agrees with the  Euler-Rodrigues formula in \eqref{e-FR} for $\theta\leqslant0$. $\square$
\end{proof}

\begin{corollary} 
The family $\{A(\mathbf{\hat{n}},\theta);\theta\in\mathbb{R}\}$, where
\begin{eqnarray*}
&&A(\mathbf{\hat{n}},\theta)=\\
&&\left[\begin{matrix}
n_1^2(1-\cos(\theta))+\cos(\theta) & n_1n_2(1-\cos(\theta))-n_3\sin(\theta) & n_1n_3(1-\cos(\theta))+n_2\sin(\theta)\\n_1n_2(1-\cos(\theta))+n_3\sin(\theta) & n_2^2(1-\cos(\theta))+\cos(\theta) & n_2n_3(1-\cos(\theta))-n_1\sin(\theta)\\ n_1n_3(1-\cos(\theta))-n_2\sin(\theta) & n_2n_3(1-\cos(\theta))+n_1\sin(\theta) & n_3^2(1-\cos(\theta))+\cos(\theta)
\end{matrix}\right]
\end{eqnarray*}
is a uniformly continuous group on $\mathbb{R}^3$ with infinitesimal generator $G:\mathbb{R}^3\to\mathbb{R}^3$ given by
\[
G=\left[\begin{matrix}
0 & -n_3 & n_2 \\ n_3 & 0 & -n_1 \\ -n_2 & n_1 & 0
\end{matrix}\right].  
\]
\end{corollary}

\begin{proof}
That family $\{A(\mathbf{\hat{n}},\theta);\theta\in\mathbb{R}\}$ is a group is an immediate consequence of the definition of $A(\mathbf{\hat{n}},\theta)$ in \eqref{eqrot}. We obtain $G$ easily from the definition of infinitesimal generator of a group 
\begin{equation*}
D(G)=\left\{u\in \mathbb{R}^3; \lim_{\theta\to 0}\frac{A(\mathbf{\hat{n}},\theta)u-u}{\theta}\ \text{exists}\right\}
\end{equation*}
and
\begin{equation*}
Gu=\lim_{\theta\to 0}\frac{A(\mathbf{\hat{n}},\theta)u-u}{\theta},\ \text{for any}\ u\in D(G).
\end{equation*}
Since $G$ is a bounded linear operator, we conclude that $\{A(\mathbf{\hat{n}},\theta);\theta\in\mathbb{R}\}$ is a uniformly continuous group on $\mathbb{R}^3$. $\square$
\end{proof}

\begin{remark}
In particular, we can obtain the explicit expression of the logarithm of rotations $A(\mathbf{\hat{n}},\theta)$ thanks to the fact that the logarithm is the infinitesimal generator of the uniformly continuous group $\{A^\alpha(\mathbf{\hat{n}},\theta);\alpha\in\mathbb{R}\}$ on $\mathbb{R}^3$; namely, we have
\[
\log A(\mathbf{\hat{n}},\theta)=\left[\begin{matrix}
0 & -\theta n_3 & \theta n_2 \\ \theta n_3 & 0 & -\theta n_1 \\ -\theta n_2 & \theta n_1 & 0
\end{matrix}\right].
\]
\end{remark}

\end{document}